\documentclass{amsart}
\usepackage{amsfonts,amssymb,amsmath,a4wide}
\usepackage[active]{srcltx}
\usepackage[british]{babel}
\usepackage{hyperref}

\newcommand{\qedbox}{ \fbox{}}
\newenvironment{prooftheodiffeo}{\noindent\emph{Proof of Theorem
  \ref{diffeo}: }}{\hfill$\qedbox$}

\newenvironment{proofoflem}{\noindent\textsc{Proof of the lemma \ref{majzinf}: }}{\hfill$\qedbox$}
\newenvironment{prooftheo}{\noindent\emph{Proof of Theorems \ref{theo14} and \ref{15}: }}{\hfill$\qedbox$}

\newenvironment{proofstab}{\noindent\emph{Proof of Theorem \ref{stab}: }}{\hfill$\qedbox$}
\newenvironment{prooflambda}{\noindent\emph{Proof of Corollary \ref{lambda}: }}{\hfill$\qedbox$}

\newtheorem{theo}{Theorem}[section]
\newtheorem{lem}{Lemma}[section]
\newtheorem{pro}{Proposition}[section]

\newtheorem{rema}{Remark}[section]
\newtheorem{cor}{Corollary}[section]

\def\R{\mathbb{R}}
\def\S{\mathbb{S}}

\def\insm{\displaystyle\int_{M}} 
\def\vol{dv}
\def\volh{dv_h}
\def\Vol{V(M)}
\def\svar{(M^m,g)}
\def\var{(M^n,g)}
\def\amb{(N^{n+1},h)}
\def\la{\lambda_1(M)}
\def\lap{\lambda_1(M)}
\def\pinch{(I_{p,\varepsilon})}

\def\hc{h^2}

\def\sdl{s_{\delta}}
\def\cdl{c_{\delta}}

\def\cmu{c_{\mu}}
\def\intsdc{\displaystyle\insm \sdl^2\vol}
\def\intcdc{\displaystyle\insm \cdl^2\vol}
\def\nhinf{\|H\|_{\infty}}

\def\nhinfc{\|H\|^2_{\infty}}

\def\xt{X^T}
\def\xd{|X|^2}
\def\dxtd{\delta|\xt|^2}
\def\hxnu{H\langle X, \nu\rangle}
\def\gradm{\nabla^M }
\def\gradn{\nabla^N }

\def\delx{\delta X}

\def\xsrx{\frac{X}{|X|}}
\def\dvgce{\text{div\ }}
\def\gradmrc{|\gradm r|^2}
\def\intxd{\insm\xd\vol}
\def\nldx{\|X\|_2^2}
\def\nldxpc{\|X\|_2}
\def\nldxt{\|\xt\|_2^2}
\def\nldy{\|Y\|_2^2}
\def\nldw{\|W\|_2^2}
\def\valx{|X|}
\def\valxpidm{\valx^{1/2}}
\def\ninfx{\|X\|_{\infty}}

\def\nlups{\|\psi\|_1}

\def\ray{\sdl^{-1}\left(\frac{1}{h}\right)}
\def\sm{s_{\mu}}
\def\cm{c_{\mu}}

\def\rbfi{\overline{R}^{\phi}}
\def\Ric{\text{Ric}}

\begin{document}\larger[2]
\title[]{EIGENVALUE PINCHING AND APPLICATION TO THE STABILITY AND THE ALMOST UMBILICITY OF HYPERSURFACES}



\author[J.-F. Grosjean, J.Roth]{J.-F. Grosjean, J.Roth}

\address[J.-F. Grosjean]{Institut \'Elie Cartan (Math\'ematiques), Universit\'e Henri Poinca\-r\'e Nancy I, B.P. 239, F-54506 Vand\oe uvre-les-Nancy cedex, France}
\email{grosjean@iecn.u-nancy.fr}

\address[J. Roth]{LAMA, Universit\'e Paris-Est - Marne-la-Vall\'ee, 5 bd Descartes, Cit\'e Descartes, Champs-sur-Marne, F-77454 Marne-la-Vall\'ee }
\email{Julien.Roth@univ-mlv.fr}

\date{\today}

\begin{abstract} In this paper we give pinching theorems for the first nonzero eigenvalue of the Laplacian on compact hypersurfaces of ambient spaces with bounded sectional curvature. As an application we deduce a rigidity result for stable constant mean curvature hypersurfaces $M$ of these spaces $N$. Indeed, we prove that if $M$ is included in a ball of radius small enough then the Hausdorff-distance between $M$ and a geodesic sphere $S$ of $N$ is small. Moreover $M$ is diffeomorphic and quasi-isometric to $S$. As other application, we obtain rigidity results for almost umbilic hypersurfaces.    

\end{abstract}

\maketitle
\tableofcontents

\section{Introduction}
One way to show that the geodesic spheres are the only stable constant mean curvature hypersurfaces of classical space forms (i.e. Euclidean space, spherical space and hyperbolic space) is to prove that there is equality in the well-known Reilly's inequality. One of the main points of the present paper is to obtain new stability results for hypersufaces immersed in more general ambient spaces by using a Reilly's inequality proved by Heintze (\cite{hei}). Note that the iso\-perimetric problem is a particular case of these stability results proved here. Indeed  compact stable constant mean curvature hypersurfaces bounding a domain appear as solution of the isoperimetric problem. We know that solutions of this problem exist on any compact Riemannian manifold and are smooth possibly up to a singular set of codimension at least $8$ (see theorem 1 of \cite{ros2}, see also \cite{mor} and \cite{nar}). Moreover, in any dimension, smooth solutions exist in a neighborhood of non-degenerate critical point of the scalar curvature (\cite{ye}).

First, let us recall Reilly's  inequality. Let $\svar$ be a compact, connected and oriented $m$-dimensional Riemannian
ma\-nifold without boundary isometrically immersed by $\phi$ in the simply connected space form $N^{n+1}(c)$ ($c=0$, $1$ ,$-1$ respectively for Euclidean space, sphere or hyperbolic space). Reilly's inequality gives an extrinsic upper bound for the first nonzero eigenvalue $\lap$ of the Laplacian of $\svar$ in term of the square of the length of the mean curvature $H$. Indeed we have
\begin{align}\label{reilly}\lap\leqslant\frac{m}{\Vol}\insm(|H|^2+c)\vol\end{align}
where $\vol$ and $V(M)$ denote respectively the Riemannian volume element and the volume of $\svar$. Moreover in the case of hypersurfaces (i.e. $m=n$), equality holds if and only if $\var$ is  immersed as a geodesic sphere of $N^{n+1}(c)$. For $c=0$ this inequality was proved by Reilly (\cite{rei}) and can ea\-sily be extended to the spherical case $c=1$ by considering the canonical embedding of $\S^n$ in $\R^{n+1}$. For $c=-1$ it has been proved by El Soufi and Ilias in \cite{elsili}. 

In the sequel we will consider a weaker inequality due to Heintze (\cite{hei}) which generalizes the previous one for the case where $\svar$ is isometrically immersed by $\phi$ in a $(n+1)$-dimensional Riemannian ma\-nifold $\amb$ whose sectional curvature $K^N$ is bounded above by $\delta$. Indeed if $\phi(M)$ lies in a convex ball and if the radius of this ball is $\frac{\pi}{4\sqrt{\delta}}$ in the case $\delta>0$, we have
\begin{align}\label{heintze}\lap\leqslant m(\nhinf^2+\delta)
\end{align}
where $\|H\|_{\infty}$ denotes the $L^{\infty}$-norm of the mean curvature. Now for $m=n$ if we assume that $K^N$ is bounded below by $\mu$ and $M$ has a constant mean curvature $H$ and is stable (see section 5) we have

$$n(H^2+\mu)\leqslant\lap\leqslant n(H^2+\delta)$$
Consequently we see that if $N$ is not of constant sectional curvature we can't conclude as in the case of space forms. However, the above inequality is a kind of pinching on Reilly's inequality, that is a condition of almost equality. Such conditions have been studied for Reilly's inequality in Euclidean space in \cite{colgro}. In the present paper we will generalize  the results of \cite{colgro} to the inequality (\ref{heintze}) for hypersurfaces (i.e. $m=n$) of ambient spaces with non constant sectional curvature. That amounts to finding conditions on geometric invariants so that if we have
$$(\Lambda_{\varepsilon})\hspace{1cm} n(\nhinfc+\delta)<\la(1+\varepsilon)$$
then $M$ is close to a sphere in a certain sense.

This problem is a particular case of a pinching concerning the moment of inertia $J_p(M)$ of $M$ with respect to a point $p$. It is defined by
$$J_p(M):=\nldxpc$$
where $X_x:=\sdl(r(x))\gradn r\mid_x$, $r(x)$ is the geodesic distance between $x$ and $p$, $\|\ \cdot\ \|_q$ is the $L^q$-norm on $C^{\infty}(M)$ defined by $\|f\|_q^q=\frac{1}{V(M)}\insm|f|^q\vol$ and $\sdl$ is the function defined by
\[ \sdl(r)=\left\{ \begin{array}{lll}

\frac{1}{\sqrt{\delta}}\sin\sqrt{\delta}r &\; \text{if} \;\; \delta>0\\

r &\; \text{if}\;\; \delta=0 \\

\frac{1}{\sqrt{|\delta|}}\sinh\sqrt{|\delta|}r &\; \text{if} \;\; \delta<0,\ \ \end{array} \right.\] 
The invariant $J_p(M)$ satisfies the inequality
\begin{align}\label{inertiamom}1\leqslant(\nhinf^2+\delta)J_p(M)^2\end{align}
where $\phi(M)$ is contained in the ball of center $p$ and of radius $\frac{\pi}{4\sqrt{\delta}}$ if $\delta>0$. We associate to this inequality the pinching
$$(I_{p,\varepsilon})\hspace{1cm}(\nhinf^2+\delta)J_p(M)^2\leqslant 1+\varepsilon$$
If $(\Lambda_{\varepsilon})$ holds and $\varepsilon$ is small enough ($\varepsilon<1/4$) then for the center of mass $p_0$ of $M$, $(I_{p_0,6\varepsilon})$ is satisfied (see proposition \ref{IimpliqueL}).

In fact, the pinching $(I_{p,\varepsilon})$ is more general than this one concerning the extrinsic radius $R(M)$ defined as the radius of the smallest ball containing $\phi(M)$. We recall that we have the following lower bound of the radius (see \cite{baikou} for instance)
\begin{align}\label{boundradius1}
\frac{\sdl(R(M))}{\cdl(R(M))}\geqslant\frac{1}{\nhinf}
\end{align}
where $\cdl=\sdl'$ or equivalently
\begin{align}\label{boundradius2}
1\leqslant(\nhinfc+\delta)\sdl(R(M))^2
\end{align}
In the case of hypersurfaces of the space form of curvature $\delta$, equality in \eqref{boundradius2} cha\-racterizes geodesic spheres. The associated pinching 
$$(R_{\varepsilon})\hspace{1cm}(\nhinfc+\delta)\sdl(R(M))^2\leqslant (1+\varepsilon).$$
has been treated for hypersurfaces of ambient spaces with constant sectional curvature in \cite{rot}. It is easy to see that if $(R_{\varepsilon})$ holds, then $(I_{p_0,\varepsilon})$ is satisfied for the center $p_0$ of the ball of radius $R(M)$ containing $\phi(M)$.

Before giving the main theorems, we set some notations which will be more convenient. Throughout the paper, we will let $h=(\nhinfc+\delta)^{1/2}$ and use $B$ to denote the second fundamental form. Moreover we will let $B(p,R)$ the geodesic ball in $N$ of center $p$ and radius $R$.

We will need two hypotheses on the volume of $M$ and on the injectivity radius $i(N)$ coming from hypotheses assumed in a result on a Sobolev inequality due to Hoffman and Spruck (\cite{hofspr1} and \cite{hofspr2}). Indeed we will assume that $i(N)\geqslant\frac{\pi}{\sqrt{\delta}}$ if $\delta>0$ and we will consider ${\mathcal H}_V(n,N)$ the space of all Riemannian compact, connected and oriented $n$-dimensional Riemannian ma\-nifolds without boun\-dary isome\-trically immersed by \ \ $\phi$\ \  in\ \ $\amb$\ \ which\ \ satisfy\ \ the\ \ following\ \ hypothesis\ \ on\ \ the\ \ volume : $\Vol\leqslant\frac{c \omega_n}{\delta^{n/2}}$ if $\delta>0$ and $\Vol\leqslant c \omega_n i(N)^n$ if $\delta\leqslant 0$ for some constant $c$. For convenience we take $1/\sqrt{\delta}=+\infty$ if $\delta\leqslant 0$.

Let us state the first main theorem.
\begin{theo}\label{diffeo} Let $\amb$ be a $n+1$-dimensional Riemannian ma\-nifold whose sectional curvature $K^N$ satisfies $\mu\leqslant K^N\leqslant \delta$  and $i(N)\geqslant\frac{\pi}{\sqrt{\delta}}$ if $\delta>0$. Let $M\in {\mathcal H}_V(n,N)$ and $p$ be a point of $N$ such that $\phi(M)\subset B\left(p,\min\left(\frac{\pi}{4\sqrt{\delta}},i(N)\right)\right)$. Let $\varepsilon<1$, $q>n$ and  $A>0$. Let us assume that $\max(\Vol^{1/n}\nhinf,\Vol^{1/n}\|B\|_q)\leqslant A$ for\ \ $\delta\geqslant 0$ \ \ (resp.\ \ $\max(\Vol^{1/n}\nhinf,\frac{\nhinf}{h},\Vol^{1/n}\|B\|_q)\leqslant A$\ \ for\ \ $\delta<0$). Then there exist positive constants $C:=C(n,q,A)$, $\alpha:=\alpha(q,n)$ such that if $(I_{p,\varepsilon})$ holds, $\varepsilon^{\alpha}<1/C$ and $\phi(M)$ is contained in the ball $B\left(p,\sdl^{-1}\left(\sqrt{\frac{\varepsilon}{\delta-\mu}}\right)\right)$ then
$$d_{H}\left(\phi(M), S\left(p,\sdl^{-1}\left(\frac{1}{h}\right) \right)\right)\leqslant\frac{C}{h}\varepsilon^{\alpha}$$
where $d_H$ denotes the Hausdorff distance. Moreover $M$ is diffeomorphic and $\varepsilon^{\alpha}$-quasi-isometric to $S(p,\sdl^{-1}\left(\frac{1}{h}\right) )$. Namely there exists a diffeomorphism from $M$ into $S(p,\sdl^{-1}\left(\frac{1}{h}\right) )$ so that
$$\left||dF_x(u)|^2-1\right|\leqslant C\varepsilon^{\alpha}$$
for any $x\in M$, $u\in T_x M$ and $|u|=1$.
\end{theo}
We recall that the Hausdorff distance between two compact subsets $A$ and $B$ of a metric space is given by
$$d_{H}(A,B)=\inf\{A\subset V_{\eta}(B) \ \text{and}\  B\subset V_{\eta}(A)\}$$
where  for any subset $A$, $V_{\eta}(A)$ is the tubular neighborhood of $A$ defined by $V_{\eta}(A)=\{x | d(x,A)<\eta\}$.

As in the euclidean case (see \cite{colgro}) for the pinching of $\la$ or as in the hyperbolic case or spherical case (\cite{rot}) for the pinching of extrinsic radius we can obtain the Hausdorff proximity strictly with a dependence on $\nhinf$.

Obviously we have the following corollary
\begin{cor}\label{lambda} Let \ \ $\amb$\ \  be a\ \  $n+1$-dimensional Riemannian ma\-nifold whose sectional curvature \ \ $K^N$\ \  satisfies\ \  $\mu\leqslant K^N\leqslant \delta$\ \   and\ \  $i(N)\geqslant\frac{\pi}{\sqrt{\delta}}$\ \  if\ \  $\delta>0$. Let\ \  $M\in {\mathcal H}_V(n,N)$. Let us assume that\ \  $\phi(M)$\ \  lies in a convex ball of radius\ \  $\min\left(\frac{\pi}{8\sqrt{\delta}},\frac{i(N)}{2}\right)$. Let\ \  $p_0$\ \  be the center of mass of\ \  $M$. Let\ \  $\varepsilon<1/6$\ ,\ \  $q>n$\ \  and\ \  $A>0$.  Let us assume that $\max(\Vol^{1/n}\nhinf,\Vol^{1/n}\|B\|_q)\leqslant A$ for\ \ $\delta\geqslant 0$ \ \ (resp.\ \ $\max(\Vol^{1/n}\nhinf,\frac{\nhinf}{h},\Vol^{1/n}\|B\|_q)\leqslant A$\ \ for\ \ $\delta<0$). Then there exist positive constants $C:=C(n,q,A)$, $\alpha:=\alpha(q,n)$ such that if $(\Lambda_{\varepsilon})$ holds, $\varepsilon^{\alpha}<1/C$ and $\phi(M)$ is contained in the ball $B\left(p_0,\sdl^{-1}\left(\sqrt{\frac{\varepsilon}{\delta-\mu}}\right)\right)$ then
$$d_{H}\left(\phi(M), S\left(p_0,\sdl^{-1}\left(\frac{1}{h}\right) \right)\right)\leqslant\frac{C}{h}\varepsilon^{\alpha}$$
and $M$ is diffeomorphic and $\varepsilon^{\alpha}$-quasi-isometric to $S(p_0,\sdl^{-1}\left(\frac{1}{h}\right) )$.
\end{cor}
Theorem \ref{diffeo} allows to obtain an application for the stable constant mean curvature hypersurfaces. Indeed we have the following stability theorem
\begin{theo}\label{stab}  Let $\amb$ be a $n+1$-dimensional Riemannian ma\-nifold whose sectional curvature $K^N$ satisfies $\mu\leqslant K^N\leqslant \delta$ and $i(N)\geqslant\frac{\pi}{\sqrt{\delta}}$ if $\delta>0$ and let $M\in {\mathcal H}_V(n, N)$. Let us assume that\ \  $\phi(M)$\ \  lies in a convex ball of radius\ \ $\min\left(\frac{\pi}{8\sqrt{\delta}},\frac{i(N)}{2}\right)$. Let\ \  $p_0$\ \  be the center of mass of\ \  $M$. Let\ \ $\varepsilon<1/6$,\ \ $q>n$ and \ \ $A>0$. Then there exist positive constants $C:=C(n,q,A)$, $\alpha:=\alpha(q,n)$ and $R(\delta,\mu,\varepsilon)$ such that if $\phi$ is of constant mean curvature $H$ and stable, $\Vol^{1/n}\|B\|_q\leqslant A$\ \ for\ \ $\delta\geqslant 0$ \ \ (resp.\ \ $\max(\frac{H}{h},\Vol^{1/n}\|B\|_q)\leqslant A$\ \ for\ \ $\delta<0$), $\varepsilon^{\alpha}<1/C$ and $\phi(M)$ is contained in a convex ball of radius  $\frac{1}{2}\sdl^{-1}\left(\sqrt{\frac{\varepsilon}{2(\delta-\mu)}}\right)$ then
$$d_{H}\left(\phi(M), S\left(p_0,\sdl^{-1}\left(\frac{1}{h}\right) \right)\right)\leqslant\frac{C}{h}\varepsilon^{\alpha}$$
and $M$ is diffeomorphic and $\varepsilon^{\alpha}$-quasi-isometric to $S(p_0,\sdl^{-1}\left(\frac{1}{h}\right) )$.\end{theo}
\begin{rema} Note that from (\cite{morjoh}) we know that stable constant mean curvature embedded hypersurfaces bounding a small volume are nearly round spheres. In our corollary we consider the more general case of immersed hypersurfaces. Moreover we give a proximity with a geodesic sphere of the ambient space with explicite center and radius.

On the other hand let us recall a result concerning the topology of isoperimetric hypersurfaces (the embedded case). For instance if $n=2$ and the Ricci curvature $\Ric^N$ of $N$ is bounded below by $2$, Ros proved that if the volume of $M$ is large enough then $M$ is homeomorphic either to a sphere or a torus (\cite{ros1}).
\end{rema}
As another application of theorems \ref{diffeo} we have results for the almost umbilic hypersurfaces of space forms. These theorems are to be compared with results of Shiohama and Xu (\cite{shixu1} and \cite{shixu2}) who obtain conditions on the Betti numbers.
\begin{theo}\label{theo14}  Let $\amb$ be a $(n+1)$-dimensional Riemannian ma\-nifold with constant sectional curvature $\delta\neq 0$ and let $M\in {\mathcal H}_V(n,N)$. Let us assume that $\phi(M)$ lies in a convex ball of radius $\frac{\pi}{8\sqrt{\delta}}$. Let $p$ be the center of mass of $M$. Let $\varepsilon<1$, $r,q>n$ and $A>0$. Moreover let us assume that $\max(\Vol^{1/n}\nhinf,\Vol^{1/n}\|B\|_q)\leqslant A$ for\ \ $\delta\geqslant 0$ \ \ (resp.\ \ $\max(\Vol^{1/n}\nhinf,\frac{\nhinf}{h},\Vol^{1/n}\|B\|_q)\leqslant A$\ \ for\ \ $\delta<0$). Then there exist positive constants \ \ $C:=C(n,q,A)$\ ,\ \  $\alpha:=\alpha(q,n)$\ \  such that if \ \ $\varepsilon^{\alpha}\leqslant 1/C$ and 
\begin{enumerate}
\item $\|\tau\|_r\leqslant\|H\|_r\varepsilon$.
\item $\left\|H^2-\nhinf^2\right\|_{r/2}\leqslant\|H\|_r^2\varepsilon$.
\end{enumerate}
Then \ \ $M$\ \ is\ \ $\varepsilon$-Hausdorff close, diffeomorphic and\ \ $\varepsilon$-quasi-isometric to\ \ $S(p,\sdl^{-1}\left(\frac{1}{h}\right) )$. 
\end{theo}
\begin{rema} The dependence on $\|B\|_q$ is not necessary for the Hausdorff proximity. 
\end{rema}
In the Euclidean case, using the pinching theorem proved in \cite{colgro} we can improve the condition 2)
\begin{theo}\label{15} Let $\var$ be a compact, connected and oriented $n$-dimensional Riemannian ma\-nifold without boundary
isometrically immersed by $\phi$ in $\R^{n+1}$. Let $p$ be the center of mass of $M$. Let $\varepsilon<1$, $r,q>n$, $s\geqslant r$ and $A>0$. Let us assume that $\Vol^{1/n}\|H\|_q\leqslant A$. Then there exist positive constants $C:=C(n,q,A)$, $\alpha:=\alpha(q,n)$ such that if $\varepsilon^{\alpha}\leqslant 1/C$ and 
\begin{enumerate}
\item $\|\tau\|_r\leqslant\|H\|_r\varepsilon$.
\item $\left\|H^2-\|H\|_s^2\right\|_{r/2}\leqslant\|H\|_r^2\varepsilon$.
\end{enumerate} 
Then $M$ is $\varepsilon$-Hausdorff close to $S\left(p,\frac{1}{\|H\|_2} \right)$. Moreover if $\Vol^{1/n}\|B\|_q\leqslant A$ then $M$ is diffeomorphic and $\varepsilon$-quasi-isometric to $S\left(p,\frac{1}{\|H\|_2} \right)$.
\end{theo}
\section{Preliminaries}
Let $\var$ be a compact, connected $n$-dimensional Riemannian manifold
isometrically immersed by $\phi$ in an $(n+1)$-dimensional Riemannian manifold $\amb$ whose sectional curvature is bounded by $\delta$. For any point $p\in N$ let us consider $\exp$ be the exponential map at this point. Locally we consider
$(x_i)_{1\leqslant i\leqslant n}$ the normal coordinates of $N$ centered at $p$ and
for all $x\in N$, we denote by $r(x)=d(p,x)$, the geodesic distance between $p$ and $x$ on $\amb$. 

We recall that the function $\cdl$ is defined by $\cdl=\sdl'$. Obviously, we have $\cdl^2+\delta\sdl^2=1$ and $\cdl'=-\delta\sdl$. 

The gradient of a function $u$ defined on $N$ with respect to $h$ will be denoted by $\gradn u$ and the gradient with respect to $g$ of the restriction of $u$ on $M$ will be denoted by $\gradm u$.

Now considering the vector field on $M$, $X=\sdl\gradn r$ we recall that Heintze proved that
\begin{align}\label{div}\dvgce(\xt)\geqslant n\cdl-n\hxnu\end{align}
Then using this identity
\begin{align*}\insm(n-\dxtd)\vol&=\insm(n-\dvgce(\xt)\cdl)\vol\\
&\leqslant \insm(n-n\cdl^2+n\hxnu\cdl)\vol\displaybreak[2]\\
&=\insm(n\delta\sdl^2+n\hxnu\cdl)\vol\displaybreak[2]\\
&\leqslant\insm n\delta\sdl^2\vol+\nhinf\insm n\sdl\cdl\vol\displaybreak[2]
\end{align*}
and using again (\ref{div}) we get
\begin{align*}\insm(n-\dxtd)\vol&\leqslant n\delta\insm\xd\vol+\nhinf\insm(n \hxnu\sdl+\dvgce(\xt)\sdl)\vol\\
&=n\delta\intxd+\nhinf\insm(n \hxnu\sdl-\cdl\sdl\gradmrc)\vol\\
&\hspace{-2cm}\leqslant n\delta\intxd+n\nhinfc\insm|\langle X,\nu\rangle||X|\vol-\nhinf\insm\cdl\sdl\gradmrc\vol\\
&\hspace{-2cm}\leqslant n\delta\intxd+n\nhinfc\insm\xd\vol-\nhinf\insm\cdl\sdl\gradmrc\vol\\
&\hspace{-2cm}\leqslant n(\nhinfc+\delta)\intxd-\nhinf\insm\cdl\sdl\gradmrc\vol\end{align*}
Finally
$$1\leqslant(\nhinfc+\delta)\nldx+\frac{1}{nV(M)}\insm(\delta\sdl^2-\cdl\sdl\nhinf)\gradmrc\vol$$
Now if $\delta\leqslant 0$ the last term is nonpositif. If $\delta>0$, since we have assumed that $\phi(M)$ is in the ball $B\left(p,\frac{\pi}{4\sqrt{\delta}}\right)$, it follows that $\frac{\sdl\left(\frac{\pi}{4\sqrt{\delta}} \right) }{\cdl\left(\frac{\pi}{4\sqrt{\delta}} \right)}\geqslant\frac{1}{\nhinf}$ and then
\begin{align}\label{delta/hc}\frac{\delta}{\nhinf^2}\leqslant 1.\end{align} It follows that $\delta\sdl^2-\cdl\sdl\nhinf\leqslant 0$. This completes the proof of (\ref{inertiamom}).

Now let us recall briefly the proof of Heintze. We will use $\frac{\sdl(r)}{r}x_{i}$ as test functions in the variational cha\-racterization of $\lap$. But these functions must be $L^2$-orthogonal to the constant functions. For this purpose, we use a standard argument used by Chavel and Heintze (\cite{cha} and \cite{hei}). Indeed, if $\phi(M)$ lies in a convex ball $B$ the vector field $Y$ defined in a neighborhood of $B$ by
$$Y_{q}=\int_{M}\frac{\sdl(d(q,x))}{d(q,x)}exp_{q}^{-1}(x)\vol(x)\in
T_{q}N,\ \  q\in M\ \ ,$$
has necessarily a zero in $B$ at a point $p$ called the center of mass of $M$. Consequently, for a such $p$,
$\insm\frac{\sdl(r)}{r}x_{i}\vol=0$. For $\delta>0$, we assume in addition that $\phi(M)$ is contained in a ball of radius $\frac{\pi}{4\sqrt{\delta}}$. Indeed, in this case $\phi(M)$ lies in a ball of center $p$ (the point $p$ so that $Y_p=0$) with a radius less or equal to $\frac{\pi}{2  \sqrt{\delta}}$ and $\cdl$ is then a nonnegative function. First note that the coordinates of
$X$ in the normal local frame are
$\left(\frac{\sdl(r)}{r}x_i\right)_{1\leqslant i\leqslant n}$. Moreover Heintze has proved that $\displaystyle\sum_{i=1}^{n+1}\left|\gradm\left(\frac{\sdl(r)}{r}x_{i}\right)\right|^2\leqslant n-\dxtd$. Then from the variational characterization of the first eigenvalue and the previous proof we get
\begin{align*}\la\nldx&\leqslant\frac{1}{V(M)}\insm(n-\dxtd)\vol\\
&\leqslant n(\nhinfc+\delta)\nldx-\frac{\nhinf}{V(M)}\insm\cdl\sdl\gradmrc\vol.
\end{align*}
We end this section by the following proposition.
\begin{pro}\label{IimpliqueL} Let \ \ $\amb$\ \  be a\ \  $(n+1)$-dimensional Riemannian ma\-nifold whose sectional curvature \ \ $K^N$\ \  satisfies\ \  $K^N\leqslant \delta$. If $\phi(M)$ lies in a convex ball of radius $\min\left(i(N),\frac{\pi}{4\sqrt{\delta}}\right)$ then for any $\varepsilon<1/4$, $(\Lambda_{\varepsilon})$ implies $(I_{p_0,6\varepsilon})$ where $p_0$ is the center of mass of $M$.  
\end{pro}
\begin{proof} If $(\Lambda_{\varepsilon})$ holds then
$$n(\nhinfc+\delta)J_{p_0}(M)^2\leqslant(1+\varepsilon)\la J_{p_0}(M)^2\leqslant(n-\delta\nldxt)(1+\varepsilon)$$
If $\delta\geqslant 0$ then $(I_{p_0,\varepsilon})$ is satisfied. If $\delta<0$ from the proof of the inequality of Reilly we have 
\begin{align*}\la J_{p_0}(M)^2&\leqslant n(\nhinfc+\delta)J_{p_0}(M)^2-\frac{\nhinf}{V(M)}\insm\frac{\cdl}{\sdl}|\xt|_2^2\vol\\
&\leqslant\la J_{p_0}(M)^2(1+\varepsilon) 
\end{align*}
It follows that $\sqrt{|\delta|}\nhinf\nldxt\leqslant\la J_{p_0}(M)^2 \varepsilon$. Therefore
\begin{align*}n(\nhinfc+\delta)J_{p_0}(M)^2&\leqslant(n-\delta\nldxt)(1+\varepsilon)\\
&\leqslant\left(n+\frac{\sqrt{|\delta|}}{\nhinf}\la J_{p_0}(M)^2\varepsilon\right)(1+\varepsilon)
\end{align*}
Now noting that $\frac{|\delta|}{\nhinf^2}\leqslant 1$, $\la\leqslant n(\nhinfc+\delta)$ and $\varepsilon<1$, we get
\begin{align*}(\nhinfc+\delta)J_{p_0}(M)^2(1-2\varepsilon)\leqslant(1+\varepsilon)
\end{align*}
and if $\varepsilon<1/4$, then $(I_{p_0,6\varepsilon})$ is satisfied.
Note that if $\delta>0$, it is not necessary to suppose that $\phi(M)$ lies in a ball or radius $\frac{\pi}{4\sqrt{\delta}}$ to prove that $(\Lambda_{\varepsilon})$ implies $(I_{p_0,6\varepsilon})$.
\end{proof}
\section{An $L^2$-approach}
Throughout the paper we assume that $\phi(M)\subset B\left(p,\frac{\pi}{4\sqrt{\delta}}\right)$ for $\delta>0$. 
\begin{lem}\label{normztangent} If the pinching condition $\pinch$ holds then $\nldxt\leqslant\frac{2\varepsilon}{\nhinf^2}$.
\end{lem}
\begin{proof} We have
\begin{align*}\nldxt&=\frac{1}{V(M)}\left(\insm\xd\vol-\insm\langle X,\nu\rangle^2\vol\right)\\
&\leqslant\frac{2}{V(M)}\insm (|X|^2-|\langle X,\nu\rangle||X|)\vol\leqslant\frac{2\varepsilon}{\nhinf^2}\end{align*}
where the last inequality is coming from the proof of (\ref{inertiamom}) recalled in the preliminaries.
\end{proof}
\begin{lem}\label{estix} Let $Y=nH\cdl\nu-n\nhinfc X$. If $\pinch$ holds then $\nldy\leqslant 4n^2\nhinf^2\varepsilon$.
\end{lem}
\begin{proof} Using again (\ref{div}) and the previous lemmas we have
\begin{align*}\nldy&=\frac{n^2}{V(M)}\insm H^2\cdl^2\vol-2\frac{n^2}{V(M)}\nhinfc\insm\hxnu\cdl\vol+n^2\nhinf^4\nldx\displaybreak[2]\\
&\leqslant \frac{n^2}{V(M)}\insm H^2\cdl^2\vol+\frac{2n\nhinfc}{V(M)}\insm(\dvgce(\xt)\cdl-n\cdl^2)\vol+n^2\nhinf^4\nldx\displaybreak[2]\\
&=\frac{n^2}{V(M)}\insm H^2\cdl^2\vol+2n\delta\nhinfc\nldxt-\frac{2n^2\nhinfc}{V(M)}\insm\cdl^2\vol+n^2\nhinf^4\nldx\displaybreak[2]\\
&\leqslant -\frac{n^2}{V(M)}\nhinfc\intcdc+n^2\nhinf^4\nldx+2n\delta\nhinf^2\nldxt \displaybreak[2]\\
&=-n^2\nhinfc+n^2\nhinfc(\nhinf^2+\delta)J_p(M)^2+2n\delta\nhinf^2\nldxt \displaybreak[2]\\
&=n^2\nhinfc\varepsilon+2n\delta\nhinf^2\nldxt\displaybreak[2]\end{align*}
Now we conclude by applying the lemma \ref{normztangent} and (\ref{delta/hc}). 
\end{proof}
\begin{lem}\label{estiy} Let $W=\valxpidm\left( \delx+H\cdl\nu-h\xsrx\right)$. If $\pinch$ holds and  if $\delta\geqslant 0$, then $\nldw\leqslant6h\varepsilon$. If $\delta<0$ and $\Vol^{1/n}\nhinf\leqslant A$ then $\nldw\leqslant\left(2h+\frac{C(n)\nhinf^2}{h}A^{n/2}\right)\varepsilon$.
\end{lem}
\begin{proof} First we have
\begin{align}\label{equa}\nldw&\leqslant\frac{1}{\Vol}\insm\left( \valx|\delx+H\cdl\nu|^2-2h\langle\delx+H\cdl\nu,X\rangle+\hc\valx\right)\vol\notag\\
&\leqslant\frac{1}{\Vol}\insm\left( \valx|\delx+H\cdl\nu|^2-2h\langle\delx+H\cdl\nu,X\rangle\right)\vol+\hc\nldxpc\end{align}
Let us compute the first term
\begin{align}\label{ineg1}\frac{1}{\Vol}\insm\valx|\delx+H\cdl\nu|^2\vol&=\frac{1}{\Vol}\insm\valx\left(\delta^2\xd+2\delta\cdl\hxnu+H^2\cdl^2\right)\vol\notag\\
&\hspace{-2,5cm}=\frac{1}{\Vol}\insm\valx\left(\delta(1-\cdl^2)+H^2(1-\delta \sdl^2)H^2+2\delta\cdl\hxnu\right)\vol\displaybreak[2]\notag\\
&\hspace{-2,5cm}=\frac{1}{\Vol}\insm\valx(H^2+\delta-\delta|HX-\cdl\nu|^2)\vol\notag\\
&\hspace{-2,5cm}\leqslant\hc\nldxpc-\frac{\delta}{\Vol}\insm\valx|HX-\cdl\nu|^2\vol\displaybreak[2]\end{align}
Now let us compute the last two terms of (\ref{equa})
\begin{align*}&-\frac{2h}{\Vol}\insm\langle\delx+H\cdl\nu,X\rangle\vol+\hc\nldxpc\\
&\leqslant -\frac{2\delta h}{\Vol}\intsdc+\frac{2h}{n\Vol}\insm\dvgce(\xt)\cdl\vol-\frac{2h}{\Vol}\intcdc+\hc\nldxpc\displaybreak[2]\\
&=-2h+\frac{2h\delta}{n}\nldxt+\hc\nldxpc\displaybreak[2]\end{align*}
Therefore reporting this and (\ref{ineg1}) in (\ref{equa}), we get
$$\nldw\leqslant 2h\varepsilon+\frac{2h\delta}{n}\nldxt-\frac{\delta}{\Vol}\insm\valx|HX-\cdl\nu|^2\vol$$
Now if $\delta\geqslant 0$ then $\nldw\leqslant 6h\varepsilon$. If $\delta<0$ we have
$$\nldw\leqslant2h\varepsilon+\frac{|\delta|}{\Vol}\ninfx\insm|HX-\cdl\nu|^2\vol$$
Moreover,
$$\insm|HX-\cdl\nu|^2\vol\leqslant\nhinfc\intsdc-2\insm\hxnu\cdl\vol+\intcdc$$
Now from the proof of (\ref{inertiamom}) recalled in the preliminaries and the pinching condition we have
$$\nhinfc\intsdc-\insm\hxnu\cdl\vol\leqslant nh^2\nldx C\leqslant n\Vol \varepsilon$$
and
\begin{align*}\intcdc-\insm\hxnu\cdl\vol&\leqslant \frac{1}{n}\insm\dvgce(\xt)\cdl\vol=\delta\Vol\nldxt\end{align*}
Then we have proved that if $\delta<0$ then
\begin{align}\nldw\leqslant(2h+n|\delta|\ninfx)\varepsilon\end{align}
Now the researched inequality is a straightforward consequence of the following lemma
\begin{lem}\label{majzinf} If $\Vol^{1/n}\nhinf\leqslant A$ then $\ninfx\leqslant K(n)A^{n/2}\nldxpc$.
\end{lem}
\end{proof}
The proof of the lemma \ref{majzinf} uses a Nirenberg-Moser type of proof (see \cite{colgro}) based on a Sobolev inequality due to Hoffman and Spruck (see \cite{hofspr1}, \cite{hofspr2} and \cite{misi}) which is available under the  conditions on the injectivity radius of $N$ and the volume of $M$ contained in the definition of ${\mathcal H}_V(n,N)$.

\begin{proofoflem} Let us put $\varphi=|X|$. An easy computation shows that $|d\varphi^{2\alpha}|\leqslant 2\alpha\varphi^{2\alpha-1}\cdl$. Then if $\delta\geqslant 0$, $|d\varphi^{2\alpha}|\leqslant 2\alpha\varphi^{2\alpha-1}$. If not we have $|d\varphi^{2\alpha}|\leqslant 2\alpha\varphi^{2\alpha-1}\sqrt{1-\delta\sdl^2}\leqslant2\alpha(1+\sqrt{|\delta|}\|\varphi\|_{\infty})\varphi^{2\alpha-1}\leqslant2\alpha(1+\nhinf\|\varphi\|_{\infty})\varphi^{2\alpha-1}$. Moreover since $1\leqslant(\nhinf^2+\delta)\|X\|_2^2\leqslant2\nhinf^2\|X\|_{\infty}^2$ we deduce that for $\delta<0$ we have $|d\varphi^{2\alpha}|\leqslant 4\alpha\varphi^{2\alpha-1}\nhinf\|\varphi\|_{\infty}$. Hence, using the Sobolev inequality (see \cite{hofspr1}, \cite{hofspr2} and \cite{misi})
\begin{align}\label{simon}
\|f\|_\frac{n}{n-1}\leqslant K(n)\Vol ^\frac{1}{n}\bigl(\|df\|_1+\|Hf\|_1\bigr)
\end{align}
we get for any $\alpha\geqslant 1$ and $f=\varphi^{2\alpha}$
\begin{align*}\|\varphi\|_{\frac{2\alpha n}{n-1}}^{2\alpha}\leqslant K(n)\Vol^{1/n}2\alpha\nhinf\|\varphi\|_{\infty}\|\varphi\|_{2\alpha-1}^{2\alpha-1}
\end{align*}
Then putting $\nu=\frac{n}{n-1}$ and $\alpha=\frac{a_p+1}{2}$ where $a_{p+1}=(a_p+1)\nu$ and $a_0=2$ we have
\begin{align*}\|\varphi\|_{a_{p+1}}^{\frac{a_{p+1}}{\nu}}&\leqslant K(n)\Vol ^\frac{1}{n}(a_p+1)\nhinf\|\varphi\|_{\infty}\|\varphi\|_{a_p}^{a_p}\\
&\leqslant K(n)\Vol ^\frac{1}{n}a_p\nhinf\|\varphi\|_{\infty}\|\varphi\|_{a_p}^{a_p}
\end{align*}
Then by iterating we find
\begin{align*}\|\varphi\|_{a_{p+1}}^{\frac{a_{p+1}}{\nu^{p+1}}}&\leqslant\left(K(n)\Vol ^\frac{1}{n}a_p\nhinf\|\varphi\|_{\infty}\right)^{1/\nu^p}\|\varphi\|_{a_p}^{\frac{a_p}{\nu^p}}\\
&\leqslant\left(\prod_{i=0}^p a_i^{\frac{1}{\nu^i}}\right)\left(K(n)\Vol ^\frac{1}{n}\nhinf\|\varphi\|_{\infty}\right)^{n\left(1-\frac{1}{\nu^{p+1}}\right)}\|\varphi\|_{a_0}^{a_0}
\end{align*}
Now since $\frac{a_p}{\nu^p}$ converges to $a_0+n$ and $a_0=2$ we get
\begin{align*}\|\varphi\|_{\infty}^2\leqslant C(n)\left(\Vol ^\frac{1}{n}\nhinf\right)^n\|\varphi\|_2^2\leqslant C(n)A^n\|\varphi\|_2^2
\end{align*}
\end{proofoflem}

Let's introduce now the function $\psi=\valx^{1/2}\left| \valx-\frac{1}{h}\right|=\valx^{1/2}\left|X-\frac{1}{h}\xsrx\right|$. We will give an $L^2$-estimate of $\psi$.  
\begin{lem}\label{normldf} If $\pinch$ holds and $\delta\geqslant 0$ then 
$\nlups\leqslant\frac{C}{h^{3/2}}\varepsilon^{1/2}$. If $\delta<0$ and $\max(\Vol^{1/n}\nhinf,\frac{\nhinf}{h})\leqslant A$ then $\nlups\leqslant \frac{C(n)A^{1+n/4}}{h^{3/2}}\varepsilon^{1/2}$.
\end{lem}
\begin{proof} First we have 
\begin{align*}\psi&=|X|^{1/2}\left|\frac{1}{h^2}(h^2 X-\delta X-H\cdl\nu)+\frac{1}{h^2}\left(\delta X+H\cdl\nu-h\frac{X}{|X|}\right)\right|\\
&\leqslant\frac{|X|^{1/2}}{nh^2}|Y|+\frac{1}{h^2}|W|
\end{align*}
Then by H\"older inequality we get
$$\nlups\leqslant\frac{1}{h^2}\left(\frac{1}{n}\|X\|_2^{1/2}\|Y\|_2+\|W\|_2\right)$$
From Lemmas \ref{estix} and \ref{estiy}, we deduce easily the inequality for $\delta\geqslant 0$ and for $\delta<0$ we get
\begin{align*}\nlups&\leqslant\frac{C(n)}{h^2}\left(\frac{\nhinf}{h^{1/2}}+h^{1/2}+\frac{\nhinf}{h^{1/2}}A^{n/4}\right)\varepsilon^{1/2}\\
&\leqslant\frac{C(n)}{h^2}(Ah^{1/2}+h^{1/2}+A^{1+n/4}h^{1/2})\varepsilon^{1/2}
\end{align*}
Now from (\ref{simon}) by taking $f=1$ we see that
\begin{align}\label{Abelow}K(n)\leqslant\Vol^{1/n}\nhinf\leqslant A\end{align}
This allows us to obtain the desired inequality for $\delta<0$. 
\end{proof}
\begin{lem}\label{ninff} Let\ \ \ $\varepsilon<1$\ \ \ be\ a\ positive\ real\ number\ and\ let\ us\ assume\ that\ \ \ $\Vol^{1/n}\nhinf\leqslant A$ (resp. $\max(\Vol^{1/n}\nhinf,\frac{\nhinf}{h})\leqslant A$ for $\delta<0$). Then there exist constants $C:=C(n)$ and $\alpha:=\alpha(n)$ so that if $\pinch$ holds then
$$\|\psi\|_{\infty}\leqslant\frac{CA^{\alpha}}{h^{3/2}}\varepsilon^{\frac{1}{2(2n+1)}}$$
\end{lem}
\begin{proof} Let $\alpha\geqslant 1$ then
\begin{align*}|d\psi^{2\alpha}|&=\alpha\psi^{2\alpha-2}|d\psi^2|\\
&=\alpha\psi^{2\alpha-2}\left|\valx-\frac{1}{h}\right|\left|3\valx-\frac{1}{h}\right||d\valx|\\
&\leqslant3\alpha\psi^{2\alpha-2}\left(\ninfx+\frac{1}{h}\right)^2\cdl
\end{align*}
Proceeding as in the proof of Lemma \ref{majzinf} we find that $|d\psi^{2\alpha}|\leqslant\alpha E\psi^{2\alpha-2}$ where $E=3\left(\ninfx+\frac{1}{h}\right)^2$ if $\delta\geqslant 0$ and $E=3\left(\ninfx+\frac{1}{h}\right)^2(1+\nhinf\ninfx)$ if not. It follows that
\begin{align*}\|\psi\|_{\frac{2\alpha n}{n-1}}^{2\alpha}&\leqslant K(n)\Vol^{1/n}(\alpha E+\|\psi\|_{\infty}^2\nhinf)\|\psi\|_{2\alpha-2}^{2\alpha-2}
\end{align*}
Now we know that for $\delta\geqslant 0$, $\nhinf\leqslant h$ and $\frac{\delta}{\nhinf^2}\leqslant 1$. Moreover, $h\leqslant\nhinf$ for $\delta<0$. From these facts we deduce that
$$\|\psi\|_{\frac{2\alpha n}{n-1}}^{2\alpha}\leqslant K(n)A\alpha E'\|\psi\|_{2\alpha-2}^{2\alpha-2}$$
where $E'=\frac{E}{h}+\ninfx\left(\ninfx+\frac{1}{h}\right)^2$. Now we put $a_{p+1}=(a_p+2)\nu$ with $\nu=\frac{n}{n-1}$, $a_0=1$  and $\alpha=\frac{a_p+2}{2}$. Then noting that $\frac{a_p}{\nu^p}$ converges to $a_0+2n$, the end of the proof is similar to that Lemma \ref{majzinf} and we find
$$\|\psi\|_{\infty}^{1+2n}\leqslant K(n)(AE')^n\nlups$$
Now Lemma \ref{majzinf} and  \ref{normldf} combining with (\ref{Abelow}) allow us to conclude that $\|\psi\|_{\infty}\leqslant K(n)\frac{A^{\alpha(n)}}{h^{3/2}}\varepsilon^{\frac{1}{2(2n+1)}}$. 
\end{proof}
\begin{lem}\label{entre2sphere} Let $\varepsilon<1$ be a positive real and let us assume that $\Vol^{1/n}\nhinf\leqslant A$ (resp. $\max(\Vol^{1/n}\nhinf,\frac{\nhinf}{h})\leqslant A$ for $\delta<0$). Then there exist constants $C:=C(n)$ and $\alpha:=\alpha(n)$ so that if $\varepsilon^{\frac{1}{2n+1}}\leqslant\frac{1}{16(CA^{\alpha})^2}$ and $\pinch$ holds then
$$\left|\valx-\frac{1}{h}\right|\leqslant\frac{CA^{\alpha}}{h}\varepsilon^{\frac{1}{2(2n+1)}}\ \ \textmd{and}\ \ \left|r-\sdl^{-1}\left(\frac{1}{h}\right)\right|\leqslant\frac{CA^{\alpha}}{h}\varepsilon^{\frac{1}{2(2n+1)}}$$
\end{lem}
\begin{proof} Consider the function $f(t)=t\left(t-\frac{1}{h}\right)^2$ which is increasing on $[0, \frac{1}{3h}]$ and $[\frac{1}{h},+\infty)$ and decreasing on $[\frac{1}{3h},\frac{1}{h}]$. Then $\|\psi\|_{\infty}^2=\|f(\valx)\|_{\infty}\leqslant\frac{(CA^{\alpha})^2}{h^3}\varepsilon^{\frac{1}{2n+1}}$. If $(CA^{\alpha})^2\varepsilon^{\frac{1}{2n+1}}\leqslant\frac{1}{27}$ then $f(\valx)\leqslant\frac{1}{27h^3}<f\left(\frac{1}{3h}\right)$. Now since $\nldx\geqslant\frac{1}{h^2}$ there exists $x_0\in M$ so that $|X_{x_0}|\geqslant\frac{1}{2h}>\frac{1}{3h}$ and by connectedness of $M$, it follows that $\valx>\frac{1}{3h}$ over $M$. Then $\left|\valx-\frac{1}{h}\right|\leqslant\frac{\sqrt{3}CA^{\alpha}}{h}\varepsilon^{\frac{1}{2(2n+1)}}$. Moreover assume that $(CA^{\alpha})^2\varepsilon^{\frac{1}{2n+1}}\leqslant\frac{1}{48}$ in order to have $\frac{\sqrt{\delta}}{h}<1$ for $\delta>0$. We have
$$\left|r-\sdl^{-1}\left(\frac{1}{h} \right) \right|\leqslant\left( \sup_{I}\frac{1}{\sqrt{1-\delta y^2}}\right) \left|\valx-\frac{1}{h}\right|\leqslant\frac{3\sqrt{3}CA^{\alpha}}{h}\varepsilon^{\frac{1}{2(2n+1)}}$$
where $I=\R^+$ for $\delta\leqslant 0$ and $I=[0,\frac{4}{3\sqrt{2\delta}}]$ for $\delta>0$. We obtain the desired result by choosing the new constant $C'=3\sqrt{3}C$.
\end{proof}
\section{Proof of the diffeomorphism}
From now we will need a dependence on the second fundamental form in order to prove the diffeomorphism and the quasi-isometry. 

Let us consider $\begin{matrix}\ & \ & \ & \ & \ \\ \ & \ & \ & \ & \ \\
F & : & M & \longrightarrow & S\left(p,\ray\right)\\
\ \ & \ \ & x &\longmapsto & \exp_p \left( \ray \frac{Y}{|Y|}\right)\end{matrix}$,  where $Y=\exp_p^{-1}(x)$. For more convenience we will put $\varrho=\ray\frac{Y}{|Y|}$.
\begin{lem}\label{estidf1} Let $u\in T_x M$ so that $|u|=1$ and $v=u-\langle u,\gradm r\rangle\gradn r$. We have
$$\frac{1}{h^2 s_{\mu}(r)^2}|v|^2\leqslant|dF_x(u)|^2\leqslant \frac{s_{\mu}\left(\ray\right)^2}{\sdl(r)^2}|v|^2$$
\end{lem}
\begin{proof} An easy computation shows that
$$d\left(\frac{Y}{|Y|}\right)\left|_x(u)\right.=\frac{1}{r}d\exp_p^{-1}\left|_x(u)\right.-\frac{dr(u)}{r^2}\exp_p^{-1}(x)$$
Then we deduce that
\begin{align*}dF_x(u)&=d\exp_p\left|_{\varrho  }\right.\left( \ray d\left(\frac{Y}{|Y|}\right)\left|_x(u)\right.\right)\\
&=\frac{\ray}{r}d\exp_p \left|_{\varrho  }\right.\left(d\exp_p^{-1}\left|_x(u)\right.\right)\\
&-\frac{\ray dr(u)}{r^2}d\exp_p \left|_{\varrho  }\right.(\exp_p^{-1}(x))\\
&=\frac{\ray}{r}d\exp_p \left|_{\varrho  }\right.\left(d\exp_p^{-1}\left|_x(u)\right.\right)-\frac{\ray dr(u)}{r}\gradn r\left|_{F(x)}\right.
\end{align*}
Now let us compute the norm of $dF_x(u)$. We have
\begin{align*}|dF_x(u)|^2&=\frac{\ray^2}{r^2}\left[ \left|d\exp_p \left|_{\varrho  }\right.\left(d\exp_p^{-1}\left|_x(u)\right.\right)\right|^2\right.\\
&\left. -2\langle d\exp_p \left|_{\varrho  }\right.\left(d\exp_p^{-1}\left|_x(u)\right.\right),\gradn r\rangle_{F(x)}dr(u)+dr(u)^2\right]
\end{align*}
Now since $\exp_p$ is a radial isometry (see for instance \cite{sak}), we have
$$\langle d\exp_p \left|_{\varrho  }\right.\left(d\exp_p^{-1}\left|_x(u)\right.\right),\gradn r\rangle_{F(x)}=\langle d\exp_p^{-1}\left|_x(u)\right.,\frac{Y}{|Y|}\rangle=\langle u, \gradn r\rangle_x$$
and it follows that
\begin{align}\label{normdf}|dF_x(u)|^2=\frac{\ray^2}{r^2}\left[ \left|d\exp_p \left|_{\varrho  }\right.\left(d\exp_p^{-1}\left|_x(u)\right.\right)\right|^2-\langle\gradm r,u\rangle^2\right]\end{align}
Now
\begin{align*}\left|d\exp_p \left|_{\varrho  }\right.\left(d\exp_p^{-1}\left|_x(u)\right.\right)\right|^2&=\left| d\exp_p \left|_{\varrho  }\right.(d\exp_p^{-1}\left|_x(v)\right.)\right.\\
&\left.+\langle u,\gradm r\rangle d\exp_p \left|_{\varrho  }\right.\left( d\exp_p^{-1}\left|_x (\gradn r)\right.\right) \right|^2\end{align*}
where $v=u-\langle u,\gradm r\rangle\gradn r$. Expending this expression we get
\begin{align*}&\left|d\exp_p \left|_{\varrho  }\right.\left(d\exp_p^{-1}\left|_x(u)\right.\right)\right|^2=\\
&\left| d\exp_p \left|_{\varrho  }\right.(d\exp_p^{-1}\left|_x(v)\right.)\right|^2+\langle u,\gradm r\rangle^2\left|d\exp_p \left|_{\varrho  }\right.\left( d\exp_p^{-1}\left|_x (\gradn r)\right.\right)\right|^2\\
&+2\langle u,\gradm r\rangle\langle d\exp_p \left|_{\varrho  }\right.(d\exp_p^{-1}\left|_x(v)\right.),d\exp_p \left|_{\varrho  }\right.\left( d\exp_p^{-1}\left|_x (\gradn r)\right.\right)\rangle\\
&=\left| d\exp_p \left|_{\varrho  }\right.(d\exp_p^{-1}\left|_x(v)\right.)\right|^2+\langle u,\gradm r\rangle^2
\end{align*}
where in the last equality we have used again the radial isometry property of the exponential map. And reporting this in (\ref{normdf}) we obtain
$$|dF_x(u)|^2=\frac{\ray^2}{r^2}\left| d\exp_p \left|_{\varrho  }\right.(d\exp_p^{-1}\left|_x(v)\right.)\right|^2$$
Since $\mu\leqslant K^N\leqslant\delta$ the standard Jacobi field estimates (see for instance corollary 2.8, p 153 of \cite{sak}) say that for any vector $w$ orthogonal to $\gradn r$ at $y$ we have
$$|w|^2\frac{r^2}{\sm(r)^2}\leqslant|d\exp_p^{-1}\left|_y(w)\right.|^2\leqslant|w|^2\frac{r^2}{\sdl(r)^2}$$
This gives
$$\frac{\sdl(\ray)^2}{r^2}|d\exp_p^{-1}\left|_x(v)\right.|^2\leqslant|dF_x(u)|^2\leqslant\frac{\sm(\ray)^2}{r^2}|d\exp_p^{-1}\left|_x(v)\right.|^2$$
and applying again the standard Jacobi field estimates we obtain the desired inequalities of the lemma.
\end{proof}
From now we denote by $D$ any constant of the form $D:=c$ if $\mu\geqslant 0$ and $D:=c\cm\left(\frac{c}{h}\right)$ for some positive constant $c$.
\begin{lem}\label{estidf2} Let $u\in T_x M$ so that $|u|=1$.
$$\frac{1-\|\gradm r\|_{\infty}^2}{h^2\sdl^2(r)}\leqslant |dF_x(u)|^2\leqslant \frac{1}{h^2\sdl^2(r)}\left(1+D\left(\frac{\delta-\mu}{h^2}\right)\right)^2$$
\end{lem}
\begin{proof} Let $r\geqslant 0$. For $t\in (-\infty,\frac{\pi^2}{9r^2}]$, consider the function $\sigma_r(t)=s_{t}(r)$. An easy check yields that $\sigma_r$ is $C^1$ on $(-\infty,\frac{\pi^2}{9r^2}]$ and
\[ \sigma_r'(t)=\left\{ \begin{array}{lll}
\frac{r^3c_t(r)}{2}\left(\frac{\sqrt{t}r-\tan(\sqrt{t}r)}{(\sqrt{t}r)^3}\right) &\; \text{if}\; t\in(0,\frac{\pi^2}{9r^2}]\\

-\frac{r^3}{6}&\; \text{if} \;\; t=0 \\

\frac{r^3c_t(r)}{2}\left(\frac{-\sqrt{-t}r+\tanh(\sqrt{-t}r)}{(\sqrt{-t}r)^3}\right) &\; \text{if}\; t\in(-\infty,0) 

\end{array} \right.\]
It follows that $\sigma_r$ is decreasing on $(-\infty,\frac{\pi^2}{9r^2}]$ and that there exists a constant $E$ so that $|\sigma_r'(t)|\leqslant E r^3 c_t(r)$, for any $t\in(-\infty,\frac{\pi^2}{9r^2}]$. From this we deduce that 
\begin{align}\label{compars}0\leqslant \sm(r)-\sdl(r)\leqslant E r^3 \cm(r) (\delta-\mu)\end{align}
From the lemma \ref{entre2sphere} and from the fact that $\frac{1}{h}\leqslant\sdl^{-1}\left(\frac{1}{h}\right)$ and $\phi(M)\subset B\left(p,\frac{\pi}{4\sqrt{\delta}}\right)$ for $\delta>0$, we see that $\sdl^{-1}\left(\frac{1}{h}\right)\leqslant\frac{\pi}{3\sqrt{\delta}}$. Then applying (\ref{compars}) to $\sdl^{-1}\left(\frac{1}{h}\right)$ we obtain that
\begin{align}\label{smusdel}\sm\left(\sdl^{-1}\left(\frac{1}{h}\right)\right)\leqslant\frac{1}{h}\left(1+D\left(\frac{\delta-\mu}{h^2}\right)\right)
\end{align}
From (\ref{compars}), (\ref{smusdel}) and the lemma \ref{estidf1} we get the desired result.
\end{proof}
\begin{lem}\label{norminfxt} Let \ \ $\varepsilon<1$\ ,\ \  $q>n$ \ and\  $A$\ \ be\ \  positive\ \  real\ \  numbers. Then\ \  there \ \ \ e\-xist\  \ \ constants \ \ \ \ $C:=C(q,n)$\ \ ,\  \ \ $\alpha:=\alpha(q,n)$\ \ \ and\ \ \ $\beta:=\beta(q,n)$ \ \ \ so\ \ \ that\ \ \ if\ \ \ $\max(\Vol^{1/n}\nhinf,\Vol^{1/n}\|B\|_q)\leqslant A$ \ \  for\ \ $\delta\geqslant 0$ (resp. $\max(\Vol^{1/n}\nhinf,\frac{\nhinf}{h},\Vol^{1/n}\|B\|_q)\leqslant A$ \ \ for $\delta<0$), $\varepsilon^{\beta}\leqslant\frac{1}{CA^{\alpha}}$ and $\pinch$ holds then
$$\|\xt\|_{\infty}\leqslant \frac{CA^{\alpha}}{h}D^{\alpha}\varepsilon^{\beta}$$
\end{lem}
\begin{proof} Put $\chi=|\xt|$. Then $|d\chi^{2\alpha}|=2\alpha\chi^{2\alpha-1}(\cdl(r)|\gradm r|^2+\sdl(r)|d|\gradm r||)$. Let us estimate $|d|\gradm r||$ at a point $x$. For this, consider $(e_i)_{1\leqslant i\leqslant n}$ an orthonormal basis at $x$. We have
\begin{align*}|d|\gradm r||^2&=\frac{1}{4|\gradm r|^2}|d\langle\gradn r,\nu\rangle^2|^2\\
&=\frac{\langle\gradn r,\nu\rangle^2}{|\gradm r|^2}\sum_{i=1}^n(e_i\langle\gradn r,\nu\rangle)^2\\
&=\frac{\langle\gradn r,\nu\rangle^2}{|\gradm r|^2}\sum_{i=1}^n(\gradn dr(e_i,\nu)+B(e_i,\gradm r))^2\\
&\leqslant\frac{2}{|\gradm r|^2}\left(\sum_{i=1}^n\gradn dr(e_i,\nu)^2+|B|^2|\gradm r|^2\right)
\end{align*}
Now $\displaystyle\sum_{i=1}^n\gradn dr(e_i,\nu)^2\leqslant|\gradn dr|^2\leqslant\sum_{i=1}^{n+1}\gradn dr(u_i,u_i)$ where $(u_i)_{1\leqslant i\leqslant n+1}$ is an orthonormal basis which diagonalizes $\gradn dr$. From the comparison theorems (see for instance \cite{sak} p 153) we deduce that
\begin{align*}\sum_{i=1}^{n+1}\gradn dr(u_i,u_i)^2&\leqslant\left(\frac{c_{\mu}}{s_{\mu}}\right)^2\sum_{i=1}^{n+1}|u_i-\langle u_i,\gradn r\rangle\gradn r|^2=n\left(\frac{c_{\mu}}{s_{\mu}}\right)^2
\end{align*}
It follows that $|d|\gradm r||^2\leqslant\displaystyle\frac{2n}{|\gradm r|^2}\left(\frac{c_{\mu}}{s_{\mu}}\right)^2+2|B|^2$ and
$$|d\chi^{2\alpha}|\leqslant2\alpha\chi^{2\alpha-1}C(n)\left(\cdl+\frac{\sdl}{|\gradm r|}\left(\frac{c_{\mu}}{s_{\mu}}\right)+\sdl |B|\right)$$
Now it is easy to see that $\frac{\sdl}{s_{\mu}}$ is bounded by a constant. Then
\begin{align*}|d\chi^{2\alpha}|&\leqslant2\alpha\chi^{2\alpha-1}C(n)\left(\frac{\cdl+\cmu}{|\gradm r|}+\ninfx|B|\right)\\
&\leqslant2\alpha\chi^{2\alpha-1}C(n)\ninfx\left(\frac{\cdl+\cmu}{\chi}+|B|\right)\\
&\leqslant2\alpha\chi^{2\alpha-2}C(n)\ninfx(\cdl+\cmu+\|\chi\|_{\infty}|B|)\\
&\leqslant2\alpha C(n)\ninfx(E+\ninfx|B|)\chi^{2\alpha-2}
\end{align*}
where $E=1$ if $\mu\geqslant 0$ and $E=\cmu(r)$ if not. Now let us assume that $\alpha\geqslant 1$. Then
\begin{align*}\|\chi\|_{\frac{2\alpha n}{n-1}}^{2\alpha}&\leqslant K(n)\Vol^{1/n}2\alpha(\ninfx E\|\chi\|_{2\alpha-2}^{2\alpha-2}+\ninfx^2\|B\|_q\|\chi\|_{\frac{(2\alpha-2)q}{q-1}}^{2\alpha-2})\\
&\leqslant K(n)A2\alpha\left(\frac{\ninfx}{h}E+\ninfx^2\right)\|\chi\|_{\frac{(2\alpha-2)q}{q-1}}^{2\alpha-2}
\end{align*}
Now we put $\nu:=\frac{n(q-1)}{(n-1)q}$, $a_{p+1}:=a_p\nu+\frac{2n}{n-1}$ , $a_0=2$ and $\alpha:=\frac{1}{2}\left(\frac{q-1}{q}\right)a_p+1$. Then $a_{p+1}=\frac{2\alpha n}{n-1}$ and
\begin{align*}\|\chi\|_{a_{p+1}}^{\frac{a_{p+1}}{\nu^{p+1}}}&\leqslant\left(K(n)Aa_{p+1}\left(\frac{\ninfx}{h}E+\ninfx^2\right)\right)^{\frac{n}{n-1}\frac{1}{\nu^{p+1}}}\|\chi\|_{a_p}^{\frac{a_p}{\nu^p}}\\
&\leqslant\left(\prod_{i=1}^{p+1}a_i^{\frac{1}{\nu^i}}\right)^{\frac{n}{n-1}}\left(K(n)A\left(\frac{\ninfx}{h}E+\ninfx^2\right)\right)^{\frac{n}{n-1}\displaystyle\sum_{i=1}^{p+1}\frac{1}{\nu^i}}\|\chi\|_{a_0}^{a_0}
\end{align*}
Now noting that $\frac{a_p}{\nu^p}$ converges to $a_0+\frac{2nq}{q-n}$ we get
$$\|\chi\|_{\infty}\leqslant C(n,q)\left(A\left(\frac{\ninfx}{h}E+\ninfx^2\right)\right)^{\frac{\gamma}{2(1+\gamma)}}\|\chi\|_2^{\frac{1}{1+\gamma}}$$
where $\gamma:=\frac{nq}{q-n}$. Now combining Lemmas \ref{majzinf} and \ref{normztangent} we obtain that
$$\|\chi\|_{\infty}\leqslant C(q,n)\frac{A^{\alpha(q,n)}(E+1)^{\alpha(q,n)}}{h}\varepsilon^{\beta(q,n)}$$
Now we conclude the proof by noting that if $\mu<0$, then from Lemma \ref{entre2sphere} we have
\begin{align*}E=\cm(r)&\leqslant\cosh\left(\sqrt{\mu}\left(\sdl^{-1}\left(\frac{1}{h}\right)+\frac{CA^{\alpha}}{h}\varepsilon^{\frac{1}{2(2n+1)}}\right)\right)\\
&\leqslant\cosh\left(\sqrt{\mu}\left(\sdl^{-1}\left(\frac{1}{h}\right)+\frac{1}{4h}\right)\right)\\
&\leqslant\cm\left(\frac{c}{h}\right)
\end{align*}
\end{proof}
We can now give the proof of Theorem \ref{diffeo}.

\begin{prooftheodiffeo} From the lemma \ref{entre2sphere} we have
\begin{align*} \frac{1}{h^2\sdl^2(r)}\left(1+D\left(\frac{\delta-\mu}{h^2}\right)\right)^2-1&\leqslant\frac{\left(1+D\left(\frac{\delta-\mu}{h^2}\right)\right)^2}{(1-CA^{\alpha}\varepsilon^{\frac{1}{2(2n+1)}})^2}-1\\
&\hspace{-4cm}\leqslant\frac{4}{3}\left(2+D\left(\frac{\delta-\mu}{h^2}\right)-CA^{\alpha}\varepsilon^{\frac{1}{2(2n+1)}}\right)\left(D\left(\frac{\delta-\mu}{h^2}\right)+CA^{\alpha}\varepsilon^{\frac{1}{2(2n+1)}}\right)\\
&\hspace{-4cm}\leqslant DCA^{\alpha}\left(1+\frac{\delta-\mu}{\varepsilon h^2}\right)^2\varepsilon^{\frac{1}{2(2n+1)}}
\end{align*} 
On the other hand from Lemmas \ref{majzinf} and \ref{norminfxt} we have
\begin{align*}\|\gradm r\|_{\infty}^2\leqslant\frac{h^2}{(1-CA^{\alpha}\varepsilon^{\frac{1}{2(2n+1)}})^2}\|\xt\|_{\infty}^2\leqslant\frac{16}{9}(CA^{\alpha})^2D^{2\alpha}\varepsilon^{2\beta}
\end{align*}
Then we deduce that
\begin{align*}
\frac{1-\|\gradm r\|_{\infty}^2}{h^2\sdl^2(r)}-1&\geqslant-\frac{2(CA^{\alpha})^2D^{2\alpha}\varepsilon^{2\beta}+(1+CA^{\alpha}\varepsilon^{\frac{1}{2(2n+1)}})^2-1}{(1+CA^{\alpha}\varepsilon^{\frac{1}{2(2n+1)}})^2}\\
&\geqslant-\left(2(CA^{\alpha})^2D^{2\alpha}\varepsilon^{2\beta}+2CA^{\alpha}\varepsilon^{\frac{1}{2(2n+1)}}+(CA^{\alpha})^2\varepsilon^{\frac{1}{2n+1}}\right)\\
&\geqslant-CA^{\alpha'}(D^{2\alpha}+1)\varepsilon^{\beta'}
\end{align*}
It follows from this, the previous inequality and Lemma \ref{estidf2} that
\begin{align*}||dF_x(u)|^2-1|&\leqslant CA^{\alpha'}\max\left((D^{2\alpha}+1)\varepsilon^{\beta'},\left(1+\frac{\delta-\mu}{\varepsilon h^2}\right)^2\varepsilon^{\frac{1}{2(2n+1)}}\right)\\
&\leqslant CA^{\alpha'}D^{\alpha'}\varepsilon^{\beta''}\left(1+\frac{\delta-\mu}{\varepsilon h^2}\right)^2
\end{align*}
From (\ref{boundradius2}) we deduce that if $\phi(M)$ lies in a ball of center $p$ and radius $\sdl^{-1}\left(\sqrt{\frac{\varepsilon}{\delta-\mu}}\right)$ then $||dF_x(u)|^2-1|\leqslant CA^{\alpha'}D^{\alpha'}\varepsilon^{\beta''}$. Moreover it is easy to see that for $\delta\geqslant 0$ then $D\leqslant c\cosh c$. For $\delta<0$, we conclude by noting that $\left(\sqrt{|\mu|}\frac{c}{h}\right)^2\leqslant c^2\left(\frac{\delta-\mu}{\varepsilon h^2}\right)+c^2\frac{|\delta|}{h^2}\leqslant c^2+c^2\frac{|\delta|}{\nhinf^2}A^2\leqslant c^2\left(1+A^2\right)$.
\end{prooftheodiffeo}

\begin{prooflambda} Since $\phi(M)$ lies in a convex ball of radius $\min\left(\frac{\pi}{8\sqrt{\delta}},\frac{i(N)}{2}\right)$, we deduce that $\phi(M)\subset B(p_0,\min\left(\frac{\pi}{4\sqrt{\delta}},i(N)\right)$ where $p_0$ is the center of mass of $M$. From Proposition \ref{IimpliqueL} and the fact that fact that $\varepsilon<1/6$ we know that $(I_{p_0,6\varepsilon})$ holds. Then we can apply Theorem \ref{diffeo}.
\end{prooflambda}
\section{Application to the stability}
Briefly, we recall the problem of the stability of hypersurfaces with constant mean curvature (see for instance \cite{bardoc}).

Let $\var$ be an oriented compact $n$ -dimensional hypersurface isometrically immersed by $\phi$ in a $n+1$-dimensional oriented manifold $\amb$. We assume that $M$ is oriented by the global unit normal field $\nu$ so that $\nu$ is compatible with the orientations of $M$ and $N$. Let $F : (-\varepsilon,\varepsilon)\times M\longrightarrow N$ be a variation of $\phi$ so that $F(0,.)=\phi$. We recall that the balance volume is the function $V : (-\varepsilon,\varepsilon)\longrightarrow\R$ defined by
$$\int_{[0,t]\times M} F^{\star}\volh$$
where $\volh$ is the element volume associated to the metric $h$. It is well known that
$$V'(0)=\insm f\vol$$
where $f(x)=\langle \frac{\partial F}{\partial t}(0,x),\nu\rangle$. Moreover the area function $A(t)=\insm dv_{F_t^{\star}h}$ satisfies
$$A'(0)=-n\insm Hf\vol$$
The balance volume $V$ is said to be preserving volume if $V(t)=V(0)$ in a neighborhood of $0$ ; in this case we have $\insm f\vol=0$. Conversely, for all smooth function $f$ so that $\insm f\vol=0$, there exists a preserving volume variation so that $f=\langle \frac{\partial F}{\partial t}(0,x),\nu\rangle$. The following assertions are equivalent
\begin{enumerate}
\item The immersion $\phi$ is a critical point of the area (i.e. $A'(0)=0$ ) for all variation with preserving volume.
\item $\insm Hf\vol=0$ for any smooth function so that $\insm f\vol=0$.
\item There exists a constant $H_0$ so that $A'(0)+nH_0 V'(0)=0$ for any variation.
\item $\phi$ is of constant mean curvature $H_0$.
\end{enumerate}
An immersion with constant mean curvature $H_0$ will be said stable if $A''(0)\geqslant 0$ for all preserving volume variation. Now we consider the function $J(t)$ defined by 
$$J(t)=A(t)+nH_0 V(t)$$
Then $J''(0)$ is depending only on $f$ and we have
$$J''(0)=\insm|df|^2\vol-\insm(Ric^N(\nu,\nu)+|B|^2)f^2\vol$$
where $Ric^N$ is the Ricci curvature of $N$ with respect to the metric $h$. It is known that $\phi$ is a stable constant mean curvature immersion if and only if $J''(0)\geqslant 0$ for any smooth function so that $\insm f\vol=0$.
\begin{rema} Note that the problem which we consider is more general that the isoperimetric problem since the hypersurfaces which we consider are immersed and not necessarily embedded.
\end{rema}

Now let us give a proof of Theorem \ref{stab}.

\begin{proofstab}  Let $f$ be the first eigenfunction associated to $\lap$. Since $\insm f\vol=0$ then $J''(0)\geqslant 0$ and
$$\lap\insm f^2\vol-\insm (Ric^N(\nu,\nu)+nH^2+|\tau|^2)f^2\vol\geqslant 0$$
where $\tau$ is the umbilicity tensor (i.e. $\tau=nHg-B$). Since $\mu\leqslant K^N\leqslant\delta$, we deduce that
$$n(H^2+\mu)\leqslant\lap\leqslant n(H^2+\delta)$$
In other words, we have the pinching condition
$$nh^2\leqslant\lap\left(1+\frac{1}{\frac{h^2}{\delta-\mu}-1}\right)$$
Let $\varepsilon<1/6$. If $\phi(M)$ lies in a ball of radius $R_1:=\sdl^{-1}\left(\sqrt{\frac{\varepsilon}{2(\delta-\mu)}}\right)$ then $(\Lambda_{\varepsilon})$ is satisfied. Let $p_0$ be the center of mass of $M$. Let $R_2:=\sdl^{-1}\left(\sqrt{\frac{\varepsilon}{\delta-\mu}}\right)$ and $R_3:=\frac{1}{2}\sdl^{-1}\left(\sqrt{\frac{\varepsilon}{2(\delta-\mu)}}\right)$. Then $R_3\leqslant\min\left(R_1,\frac{1}{2}R_2\right)$. Therefore if $\phi(M)$ lies in a convex ball of radius $R_3$ we deduce that this ball contains $p_0$ and $\phi(M)\subset B(p_0,R_2)$. Moreover $(\Lambda_{\varepsilon})$ is satisfied and we conclude with Corollary \ref{lambda}.
\end{proofstab}
\section{Application to the almost umbilic hypersurfaces}
Theorems \ref{theo14} and \ref{15} are obtained by combining Theorem \ref{diffeo} and results of \cite{aubgrorot} for the Euclidean case with an eigenvalue  pinching theorem in almost positive Ricci curvature due to Aubry (\cite{aub}).
In the following theorem we denote $\underline{\Ric}(x)$ the lowest eigenvalue of the Ricci tensor $\Ric(x)$ at $x\in M$. Moreover for any function $f$, we put $f_{-}=\min(-f,0)$.
\begin{theo}({\sc Aubry})\label{Aubry} Let $\var$ be a complete $n$-dimensional Riemannian manifold and $r>n$. If $M$ has finite volume and 
$$\rho_r=\frac{1}{k\Vol^{2/r}}\left( \insm\left(\underline{\Ric}-(n-1)k\right)_{-}^{r/2}\vol\right) ^{2/r}\leqslant C(r,n)^{-2/r}$$
then $M$ is compact and $\lap\geqslant nk(1-C(r,n)\rho_r)$.
\end{theo}
\begin{prooftheo} Using Gauss formula and the fact that $N$ is of constant sectional curvature $\delta$, we have
\begin{align*}\|\Ric-(n-1)(H^2+\delta) g\|_{r/2} &=\|\rbfi+nHB-B^2-(n-1)H^2 g-(n-1)\delta g\|_{r/2}\\
&=\|(n-2)H\tau-\tau^2\|_{r/2}\\
&\leqslant(n-2)\|H\|_r\|\tau\|_r+\|\tau\|_r^2\end{align*}
Now, putting $k=\|H\|_s^2+\delta$ for $2\leqslant r\leqslant s$, we get
\begin{align*}\|\Ric-(n-1)kg\|_{r/2}&\leqslant\|\Ric-(n-1)(H^2+\delta) g\|_{r/2}+(n-1)\sqrt{n}\left\|H^2-\|H\|_s^2\right\|_{r/2}\\
&\leqslant(n-2)\|H\|_r\|\tau\|_r+\|\tau\|_r^2+(n-1)\sqrt{n}\left\|H^2-\|H\|_s^2\right\|_{r/2}
\end{align*}
If $\|\tau\|_r\leqslant\|H\|_r\varepsilon$ and $\left\|H^2-\|H\|_s^2\right\|_{r/2}\leqslant\|H\|_r^2\varepsilon$ then
$$\|\Ric-(n-1)kg\|_{r/2}\leqslant K(n)\|H\|_r^2\varepsilon$$
for $\delta\neq 0$ we choose $s=\infty$. If $\delta\geqslant 0$ and $\varepsilon\leqslant K(n,r)$ (resp. $\varepsilon\leqslant K(n)\frac{h^2}{\nhinf^2}$ for $\delta<0$) the theorem \ref{Aubry} allows us to conclude that
$$\la\geqslant n(\|H\|_s^2+\delta)(1-C(n,r)\varepsilon)$$
$$(\text{resp}.\la\geqslant n(\nhinf^2+\delta)\left(1-C(n,r)\frac{\nhinf^2}{h^2}\varepsilon\right) \text{for}\ \delta<0)$$
Now the conclusion is immediate from Corollary \ref{lambda} and \cite{aubgrorot}.
\end{prooftheo}

\end{document}